\documentclass[pdflatex,12pt,a4paper,twoside]{article}
\usepackage{authblk}
\usepackage{amsmath,amsfonts,amssymb,amsthm}
\usepackage{ifthen}
\usepackage{graphicx}
\usepackage{epsfig}
\usepackage{nicefrac}
\usepackage{mathrsfs}
\usepackage[a4paper,left=31mm,right=30mm,top=34mm,bottom=33mm]{geometry}
\usepackage{amsmath}
\usepackage{amsfonts}
\newcommand{\R}{\mathbb{R}}
\newcommand{\N}{\mathbb{N}}
\newcommand{\C}{\mathbb{C}}
\newcommand{\Z}{\mathbb{Z}}
\newcommand{\eps}{\varepsilon}
\newcommand{\fhi}{\varphi}
\newcommand{\vth}{\vartheta}

\newcommand{\del}{\partial}
\renewcommand{\div}{\mathrm{div}}
\newcommand{\curl}{\mathrm{curl}}
\newcommand{\vol}{\mathrm{Vol}}

\newcommand{\one}{\mathrm{\bf 1}}

\newcommand{\eff}{\mathrm{eff}}

\def\vectdue#1#2{\left(\begin{matrix}#1\cr #2\end{matrix}\right)}
\def\vecttre#1#2#3{\left(\begin{matrix}#1\cr #2\cr #3\end{matrix}\right)}

\def\calK{\mathcal{K}}

\renewcommand{\Re}{\mathrm{Re}}

\def\Xint#1{\mathchoice
   {\XXint\displaystyle\textstyle{#1}}%
   {\XXint\textstyle\scriptstyle{#1}}%
   {\XXint\scriptstyle\scriptscriptstyle{#1}}%
   {\XXint\scriptscriptstyle\scriptscriptstyle{#1}}%
   \!\int}
\def\XXint#1#2#3{{\setbox0=\hbox{$#1{#2#3}{\int}$}
     \vcenter{\hbox{$#2#3$}}\kern-.5\wd0}}
\def\meanint{\Xint-}

\newtheorem{theorem}{Theorem}[section]

\newtheorem{lemma}[theorem]{Lemma}
\newtheorem{proposition}[theorem]{Proposition}

\numberwithin{equation}{section}

\title{{\LARGE\bf A negative index meta-material for Maxwell's
    equations}}

\author{A.\,Lamacz, B.\,Schweizer\thanks{Technische Universit\"at
    Dortmund, Fakult\"at f\"ur Mathematik, Vogelpothsweg 87, D-44227
    Dortmund, Germany.}}

\date{July 29, 2015}

\begin{document}

\maketitle

\begin{abstract} We derive the homogenization limit for time harmonic
  Maxwell's equations in a periodic geometry with periodicity length
  $\eta>0$. The considered meta-material has a singular sub-structure:
  the permittivity coefficient in the inclusions scales like
  $\eta^{-2}$ and a part of the substructure (corresponding to wires
  in the related experiments) occupies only a volume fraction of order
  $\eta^2$; the fact that the wires are connected across the
  periodicity cells leads to contributions in the effective system. In
  the limit $\eta\to 0$, we obtain a standard Maxwell system with a
  frequency dependent effective permeability $\mu^{\eff}(\omega)$ and
  a frequency independent effective permittivity $\eps^{\eff}$. Our
  formulas for these coefficients show that both coefficients can have
  a negative real part, the meta-material can act like a negative
  index material. The magnetic activity $\mu^{\eff}\neq 1$ is obtained
  through dielectric resonances as in previous publications. The wires
  are thin enough to be magnetically invisible, but, due to their
  connectedness property, they contribute to the effective
  permittivity. This contribution can be negative due to a negative
  permittivity in the wires.
\end{abstract}


%
%

\pagestyle{myheadings} 
\thispagestyle{plain} 

\markboth{A.\,Lamacz, B.\,Schweizer}{A negative index meta-material
  for Maxwell's equations}

\section{Introduction}

Light and other electromagnetic waves exhibit interesting refraction
phenomena in meta-materials. While no natural material posesses a
negative index, by now, there exists a variety of negative index
meta-materials (a meta-material is an assembly of ordinary materials,
arranged in small substructures).  A smart choice of the microscopic
geometry can lead to macroscopic properties of the meta-material that
are not shared by any of the ordinary materials that it is made
of. Even though no natural negative index material is known, a
distribution of metallic or dielectric resonators and metallic wires
can act, effectively, as a medium with a negative index.  We present
here a periodic meta-material (periodicity length $\eta>0$) with
singular sub-structures (wires with relative radius of order $\eta$,
i.e. with absolute radius of order $\eta^2$) which is, on the one
hand, close to the experimental set-up, and, on the other hand,
accessible to a rigorous mathematical analysis.

From a modelling point of view, the Maxwell system is very simple: We
investigate solutions $(E^\eta, H^\eta)$ to the time harmonic
Maxwell's equations in three dimensions. The system contains three
parameters: the frequency $\omega\in \R$ and two material parameters,
the permeability $\mu = \mu(x)$ and the permittivity $\eps =
\eps(x)$. Since natural materials have a relative permeability close
to one, we assume that the permeability coincides with that of vacuum,
$\mu \equiv \mu_0 >0$.  All the complex behavior of the
micro-structure is encoded in one single coefficient, the relative
permittivity $\eps_\eta : \R^3\to \C$.  Here, we consider periodic
coefficients $\eps_\eta$ with large values of $|\eps_\eta|$ in the
inclusions. Using the parameter $\eps_0>0$ (permittivity of vacuum),
we investigate solutions $(E^\eta, H^\eta)$ of
\begin{align}
 \label{eq:Max1}
 \curl\, E^\eta &=\phantom{-} i\omega\mu_0 H^\eta\,,\\
 \label{eq:Max2}
 \curl\, H^\eta &=-i\omega\eps_\eta\eps_0 E^\eta\,.
\end{align}
Our aim is to analyze the behavior of the solutions $(E^\eta, H^\eta)$
in the limit $\eta\to 0$.

\smallskip The relative permittivity is given by two complex numbers
and the geometry of two types of inclusions, $\Sigma_\eta\subset \R^3$
and $\Gamma_\eta\subset \R^3$, both periodic with periodicity
$\eta$. The set $\Sigma_\eta$ consists of bulk inclusions, the number
of which is of order $\eta^{-3}$ (one inclusion in each periodicity
cell). The set $\Gamma_\eta$ represents a system of long and thin
wires, each wire has a radius of order $\eta^2$ and length of order
$1$ (the wire has a macroscopic length and a width that is small
compared to the periodicity length). With parameters $\eps_b,
\eps_w\in\C$ with $\Im(\eps_b), \Im(\eps_w)> 0$ we study the
coefficient
\begin{align}
 \label{eq:defepseta}
 \eps_\eta =
 \begin{cases}\displaystyle
   \eps_b\eta^{-2} &\text{ in } \Sigma_\eta\,,\\[1mm]
   \displaystyle
   \eps_w\eta^{-2} &\text{ in } \Gamma_\eta\,,\\[1mm]
   1 &\text{ in } \R^3\setminus (\Sigma_\eta\cup \Gamma_\eta)\,.
 \end{cases}
\end{align}

Our result is a description of the weak limit $(E, H)$ of a solution
sequence $(E^\eta, H^\eta)$. We show that an appropriately defined
pair $(\hat E, \hat H) := (E, (\hat\mu)^{-1} H)$ solves a Maxwell
system with two effective parameters $\mu^{\eff}(\omega)$ and
$\eps^{\eff}$, cp. \eqref {eq:eff1}-- \eqref {eq:eff2}.  The two
effective parameters are given by explicit formulas that involve cell
problems for the electric and for the magnetic field. In an
appropriate geometry and for an appropriate frequency $\omega$, we
find that both coefficients can have a negative real part
(simultaneously): The effective permeability $\mu^{\eff}(\omega)$ can
become negative due to resonances in the bulk inclusions
$\Sigma_\eta$.  Due to their vanishing volume fraction, the wires do
not influence the parameter $\mu^{\eff}(\omega)$. On the other hand,
due to their special topology (they form connected objects of
macroscopic length), they influence the permittivity $\eps^{\eff}$. If
$\eps_w$ has a negative real part (which is the case for many metals),
then $\eps^{\eff}$ can be negative, see \eqref {eq:eps-eff-def}.

\subsection{Literature}

Half a century ago, Veselago investigated in the theoretical study
\cite{Veselago1968} materials with negative $\eps$ and negative
$\mu$. Since the product $\eps \mu$ is positive, waves can travel in
such a medium, but surprising effects such as negative refraction and
perfect lensing can be expected. Since no natural materials exhibit a
negative $\mu$, the studies of Veselage have not been continued until
in about 2000 first ideas were published on how to construct a
negative index meta-material, see e.g.\,\cite{Pendry2004}.  Regarding
applications of negative index materials we mention the effect of
cloaking by anomalous localized resonance, see
e.g.\,\cite{BouchitteSchweizer-Cloak} and \cite{Kohn-LSW}.

A mathematical analysis of meta-materials became possible with the
development of the method of homogenization. The homogenization
technique was successfully applied in many situations, ranging from
porous media to wave equations. Also the homogenization of Maxwell's
equations has been performed. The standard result of a homogenization
process is the following: Given periodically oscillating coefficients
$a_\eta(x)$ and a corresponding solution sequence $u^\eta(x)$, every
weak limit $u$ of the solution sequence satisfies the original
equation with an averaged coefficient $a^\eff$. In the context of the
Maxwell system, the oscillating coefficients are $a_\eta(x) =
(\eps_\eta(x), \mu_\eta(x))$ and the solution is $u^\eta = (E^\eta,
H^\eta)$.  In 
\cite{MR2029130}, results of this kind are obtained for Maxwell's
equations.

Of particular interest are those homogenization results that lead to a
qualitatively different equation for the limit $u$. Examples are the
double porosity model in porous media \cite{Arbogast-1052874} or the
dispersive limit equation for waves in heterogeneous media
\cite{DLS-MR3191584, DLS-2}.  The problem at hand is similar in that
we want to combine positive index materials to obtain an effective
negative index material. Typically, such effects are obtained by
micro-structures that involve extreme parameter values and/or by
micro-structures that contain finer substructures.

In this work, we will actually use both, extreme values and fine
substructures, to obtain a surprising new limit formula in the
homogenization of Maxwell's equations. One of the pioneers in the
field is Bouchitt\'e who, together with co-authors, initiated the
field with the analysis of wire structures \cite{MR2262964},
\cite{MR1444123}. It was shown that extreme coefficient values in the
(thin) wires can lead to the effect of a negative $\eps^\eff$.

In order to obtain a negative index material, additionally a negative
$\mu^\eff$ has to be created. Several ideas have been analyzed. Based
on the analysis of a model problem, a mathematical result has been
obtained in \cite{KohnShipman}. A truely three-dimensional analysis of
a setting that is also close the some of the original designs has been
carried out in \cite{BouchitteSchweizer-Max}. In that work, it was
shown that the periodic split ring structure can lead to a negative
$\mu^\eff$. A similar result has been obtained later also for flat
rings of arbitrary shape in \cite{Lamacz-Schweizer-Max} (while the
rings in \cite{BouchitteSchweizer-Max} had to be tori).

While this seems to be the state of the art with regard to metallic
inclusions, a simpler approach using dielectric materials has been
invented in \cite{CRAS} and later developed in
\cite{BouchitteBourel2009}: Dielectric inclusions can lead to
resonances (so-called Mie-resonances) which result in an effective
magnetic activity. With resonant dielectric inclusions, a negative
$\mu^\eff$ can be obtained even without subscale variations of the
periodic geometry.

With this contribution we close a gap that has been left open by the
above mentioned works: The emphasis has always been to create either a
negative $\eps^\eff$ \cite{MR2262964, MR1444123} or a negative
$\mu^\eff$ \cite {BouchitteBourel2009, CRAS, BouchitteSchweizer-Max,
  Lamacz-Schweizer-Max} --- here we present a construction that
achieves, simultaneously, a negative $\eps^\eff$ and a negative
$\mu^\eff$ (we always refer to the real part of the coefficient). An
important point in our construction is the decoupling: $\eps^\eff$
depends only on $\eps_w$ and the shape $\Sigma_Y$ of the bulk
inclusions; given $\Sigma_Y$, we can choose $\eps_w$ to have
$\eps^\eff$ negative (independent of the frequency). Tuning the
frequency, we can generate a resonance in $\Sigma_Y$ which makes
$\mu^\eff$ negative; this process does not affect $\eps^\eff$.

\smallskip We conclude this section with a comparison of our results
to those of \cite{MR2262964}, where also the effect of thin wires has
been analyzed. In \cite{MR2262964}, the volume fraction $\theta_\eta$
of the wires is much smaller than in our study, namely $\theta_\eta
\sim \exp(-2\eta^{-2})$ (here, it is $\theta_\eta \sim \eta^2$; we
observe that the permittivity is scaled as in our work as $|\eps_\eta|
\sim \theta_\eta^{-1}$). The effective equation in \cite{MR2262964} is
a Maxwell type system that is coupled to a macroscopic equation for a
new variable (denoted by $J$ in \cite{MR2262964}). Our result is much
simpler in the sense that we obtain a standard Maxwell system in the
limit, see \eqref {eq:eff1}-- \eqref {eq:eff2}.  In this sense, our
result is very different to that of \cite{MR2262964}, despite the
similarities in the setting (thin wires with vanishing volume fraction
of macroscopic length). In order to understand how such a different
outcome is possible, let us mention the following point (somewhat
technical, but important): the tiny radii in \cite{MR2262964} make it
impossible to construct test-functions as in \eqref
{eq:psi-def}--\eqref{eq:g-def}, hence the macroscopic limit cannot be
obtained as in our setting.

\subsection{Geometry} 

The underlying idea of our approach is simple: We use dielectrical
bulk-inclusions (given by the subset $\Sigma_\eta$) as Mie-resonators,
they lead to a negative $\mu^\eff$. Additionally, we include thin
wires (given by the subset $\Gamma_\eta$) that have a negative $\eps$
(which is the case for many metals); the wires lead to a negative
$\eps^\eff$. Since the wires are thin, their effect decouples from the
Mie-resonance.

We emphasize that, defining the two sets $\Sigma_\eta$ and
$\Gamma_\eta$, the Maxwell system is completely described by
\eqref{eq:Max1}--\eqref{eq:defepseta}. Below, we consider the limit of
any bounded solution sequence and do not specify boundary
conditions. Our result is therefore applicable to any boundary value
problem.

\paragraph{Microscopic geometry.}
We construct a {\em periodic} meta-material in three space
dimensions. We start from the periodicity cube $Y = [0,1)^3$; since we
will always impose periodicity conditions on the cube $Y$, we may also
regard it as the flat torus $\mathbb{T}^3$. We now construct two
subsets, $\Sigma_Y, \Gamma_Y^\eta \subset Y$, compare Figure \ref
{F:Geometry} for an illustration.

\smallskip {\em Dielectric resonator.} The dielectric resonator is
given by a subset $\Sigma_Y \subset Y$. We assume that $\Sigma_Y$ is a
simply connected open set that has a Lipschitz boundary and that is
compactly contained in $(0,1)^3$. We call the set the dielectric
resonator since, in order to have resonances in $\Sigma_Y$, we must
assume that $\eps_b$ has a positive real part, cp.\,\eqref{eq:mu-eff}.

\smallskip {\em Thin wires.} The second subset are the (metallic)
wires. We assume that three wires $\Gamma_Y^{\eta,j}\subset Y$,
$j=1,2,3$, are contained in $Y$. We assume that each wire has a radius
of order $\eta$ inside $Y$ and that it connects two opposite sides of
$\del Y$ in a periodic fashion. We assume additionally that the wires
do not intersect each other or the resonator $\Sigma_Y$; this
assumption is not satisfied in some of the experimental designs
(``fishnet-structure''), but we do not see it as an essential property
for our method to work.

In order to make our assumptions precise (and, at the same time, to
keep the construction of special test-functions in Section \ref
{sec.macro-equations} accessible), we restrict ourself to cylindrical
wires of relative radius $\alpha>0$: For a point $(y_1^{(3)},
y_2^{(3)}) \in (0,1)^2$, the wire with index $j = 3$ has the central
line $\Gamma_Y^3 := (y_1^{(3)}, y_2^{(3)}) \times [0,1)$ and is given
by
\begin{equation}
  \label{eq:Gamma-j-eta}
  \Gamma_Y^{\eta,3} 
  := B_{\alpha\eta}((y_1^{(3)}, y_2^{(3)})) \times [0,1) \subset Y\,.
\end{equation}
The wires with indices $1$ and $2$ are constructed accordingly. The
union of the three wires is denoted by $\Gamma_Y^\eta :=
\bigcup_{j=1}^3 \Gamma_Y^{\eta,j}$.  The construction assures that the
relative radius of the wires inside the cube is $\alpha\eta$ such that
the wires vanish in the limit $\eta\to 0$. We will see that the wires
do not enter the cell problem, but, due to their connectedness across
cells, they do affect the macroscopic equations.  The real part of the
permittivity in the wires has to be negative in order to obtain a
negative quantity $\eps^\eff$ in \eqref {eq:eps-eff-def}. Since metals
can have a negative permittivity, we think of metallic wires.

\begin{figure}[h]
  \begin{center}
    \includegraphics[width=0.40\textwidth]{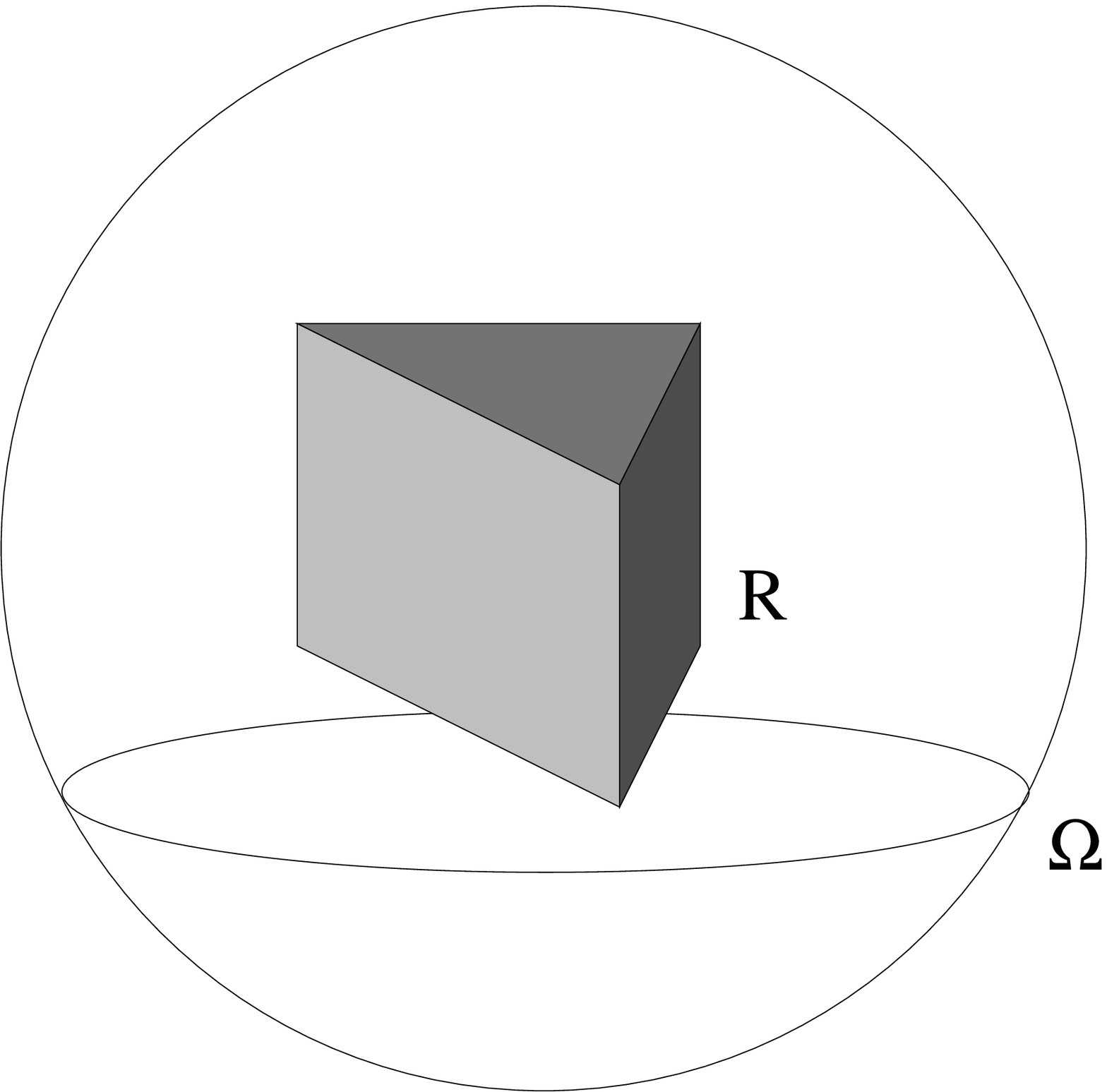}\hspace*{13mm}
    \includegraphics[width=0.47\textwidth]{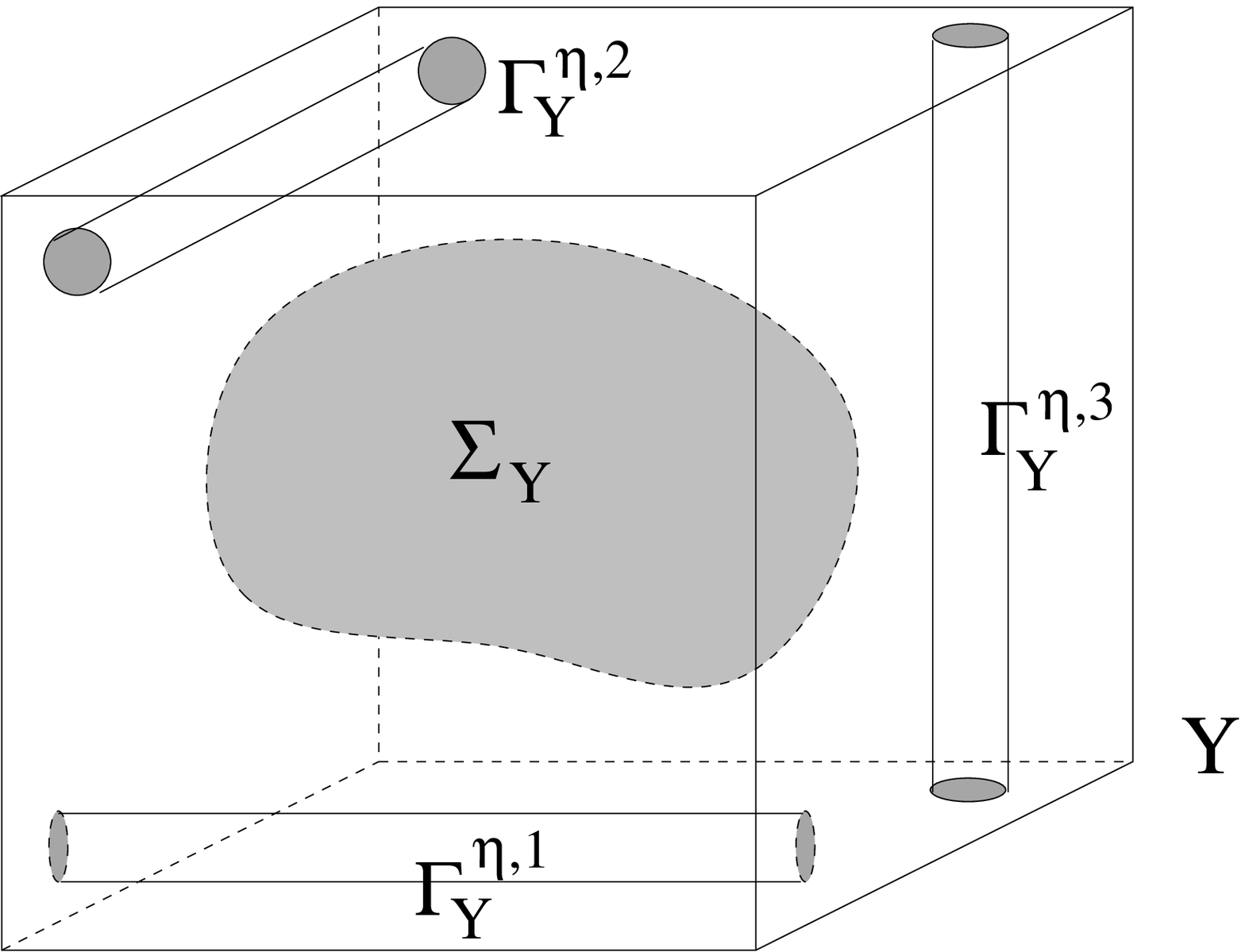}
    \caption{ {\em Left:} The macroscopic domain $\Omega$ and the
      scatterer $R\subset \Omega$ (here, $\Omega$ is a ball and $R$ is
      a prism). Maxwell's equations are solved on $\Omega$. The
      equations contain the coefficient $\eps_\eta$ which is $1$ in
      $\Omega\setminus R$ and which has large values and a
      micro-structure in $R$. Our interest is to describe the solution
      fields $(E^\eta, H^\eta)$ in the limit $\eta\to 0$. {\em Right:}
      The periodicity cell $Y$ with the subsets $\Sigma_Y$ (the
      resonator) and $\Gamma_Y^\eta$ (the wires). In the two subsets
      the coefficient $\eps_\eta$ has absolute values of order
      $\eta^{-2}$. The wires have the radius $\alpha\eta$ in the
      periodicity cell, which means that they have the radius
      $\alpha\eta^2$ in the macroscopic domain. The periodic
      construction yields connected wires with a length of order
      $\eta^0 = 1$.}
      \label{F:Geometry}
  \end{center}
\end{figure}

\paragraph{Macroscopic geometry.}
We study Maxwell's equations in an open set $\Omega\subset
\R^3$. Contained in $\Omega$ is a second domain $R\subset \Omega$ with
$\bar R\subset \Omega$. The set $R$ consists of the meta-material, on
the set $\Omega\setminus R$ we have relative permeability and relative
permittivity equal to unity, see Figure \ref {F:Geometry}, left part.

In order to define the microstructure in $R$, we use indices $k\in
\Z^3$ and shifted small cubes $Y^\eta_k := \eta (k + Y)$. We denote by
$\calK := \{k\in \Z^3 | Y^\eta_k\subset R\}$ the set of indices $k$
such that the small cube $Y^\eta_k$ is contained in $R$.  Here and in
the following, in summations or unions over $k$, the index $k$ takes
all values in the index set $\calK$. The number of relevant indices
has the order $| \calK | = O(\eta^{-3})$.

Using the local subsets $\Sigma_Y$ and $\Gamma^{\eta,j}_Y$, we can now
define the meta-material by setting
\begin{equation}
  \label{eq:slit-eta-3mac}
  \Sigma_\eta := \bigcup_{k\in \calK} \eta(k+\Sigma_Y)\,,\quad
  \Gamma_\eta := \bigcup_{k\in \calK} \bigcup_{j\in \{1,2,3\}} 
  \eta(k+\Gamma^{\eta,j}_Y)
  = \bigcup_{k\in \calK} \eta(k+\Gamma^{\eta}_Y)\,.
\end{equation}

\subsection{Main result}

With the help of cell-problems, we define two tensors $\eps^{\eff},
\mu^{\eff}(\omega)\in \C^{3\times 3}$.  The tensor $\eps^{\eff}$ is
defined in \eqref {eq:eps-eff-def} and has two contributions: The
matrix $A^\eff$ is defined with the standard electrical cell problem,
the additive term $\alpha\eps_w$ can make the real part of
$\eps^{\eff}$ negative. The effective permeability tensor $\mu^{\eff}$
is defined in \eqref {eq:HhatH-mu} through a coupled cell-problem that
appeared already in \cite {BouchitteBourel2009,
  BouchitteSchweizer-Max} and was studied there. The formula for
$\mu^{\eff}(\omega)$ is a result of the resonance properties of
$\Sigma_Y$; a negative real part of $\mu^{\eff}(\omega)$ can be
expected for frequencies that are close to eigenvalues.

\paragraph{Effective equations.} 
We derive an {\em effective} Maxwell system in $\Omega$:
\begin{align}
   \curl\ \hat E &= \phantom{-} i\omega\mu_0\, \hat\mu \hat H\,,
   \label{eq:eff1}\\
   \curl\ \hat H &= -i\omega\eps_0\, \hat\eps \hat E\,.
   \label{eq:eff2}
\end{align}
In this system, the coefficients are $\hat\mu(x) = \hat\eps(x) = 1$
for $x\in \Omega\setminus R$, whereas for $x\in R$
\begin{align*} \hat\eps(x) = \eps^{\eff}\,,\quad \hat\mu(x) =
   \mu^{\eff}(\omega)\,.
\end{align*}
A consequence of our result is the following: Since outside the
scatterer $R\subset \Omega$ the effective coefficients are equal to
$1$, the field $(\hat E,\hat H)$ coincides there with the weak limit
$(E,H)$. In particular, for $x\in \Omega\setminus R$, as $\eta\to 0$,
the solutions $(E^\eta,H^\eta)$ converge to $(\hat E,\hat H)$,
solution to the effective system \eqref {eq:eff1}--\eqref
{eq:eff2}. Let us now formulate the precise statement in our main
theorem.


\begin{theorem}\label{thm:effective-Max}
  Let $(E^\eta, H^\eta)$ be a sequence of solutions to \eqref
  {eq:Max1}--\eqref {eq:Max2}, where the coefficient $\eps_\eta$ is
  given by \eqref {eq:defepseta} and the geometry of inclusions
  $\Sigma_\eta$ and wires $\Gamma_\eta$ is as described above,
  cf.\,\eqref {eq:Gamma-j-eta} and \eqref{eq:slit-eta-3mac}.  We
  assume that the solution sequence satisfies the energy-bound
  \begin{equation}
    \label{eq:bound-improved}
    \int_{\Omega} |H^\eta|^2 + |\eps_\eta|\, |E^\eta|^2 \le C\,,
  \end{equation}
  where $C$ does not depend on $\eta$, and that $(E^\eta, H^\eta)$
  converges weakly in $L^2(\Omega)$ to a limit $(E,H)$.  Then the
  limit $(E,H)$ has the property that the modified fields $\hat E :=
  E$ and $\hat H := \hat\mu^{-1} H$ solve \eqref {eq:eff1}--\eqref
  {eq:eff2} in the sense of distributions on $\Omega$.
\end{theorem}

\medskip {\em 1. Comments on the fields $\hat E$ and $\hat H$.} The
field $\hat E$ coincides with $E$, we introduced the new name only to
have a consistent notation. Loosely speaking, the field $\hat H$ is
the limit of the fields $H^\eta$ outside the inclusions. It is
obtained from the two-scale limit $H_0(x,y)$ of $H^\eta$ by geometric
averaging, see \eqref {eq:H-normalization-rewrite}; geometric
averaging is introduced in Section \ref{ssec.geometric-av}.

\smallskip {\em 2. Easy parts of the Theorem.}  Since we assume that
the imaginary parts of $\eps_b$ and $\eps_w$ are positive, the
estimate \eqref {eq:bound-improved} implies the $L^2(\Omega)$
boundedness of $E^\eta$ and $H^\eta$. Since every weak limit is also a
distributional limit, we may therefore take the distributional limit
in \eqref {eq:Max1} and obtain $\curl\ E = i\omega\mu_0\, H$.  Because
of $\hat E := E$ and $H = \hat\mu \hat H$, we conclude \eqref
{eq:eff1} in the sense of distributions.  The challenging part of the
proof is to derive the effective equation \eqref {eq:eff2} from the
Maxwell equation \eqref {eq:Max2}.

\smallskip {\em 3. Comments on the a priori estimate
  \eqref{eq:bound-improved}.} We have formulated our main theorem with
a boundedness assumption on the solution sequence. This is done for
convenience, the discussion of the a priori estimate is available in
the literature.  We give here only a brief summary of some facts.

(a) It is sufficient to make only the (weaker) assumption of
$L^2(\Omega)$-boundedness of the solution sequence,
\begin{equation}
  \label{eq:apr-bound}
  \int_{\Omega} |H^\eta|^2 + |E^\eta|^2 \le C\,.
\end{equation}
Indeed, assuming only \eqref {eq:apr-bound} and $\Im \eps_b, \Im
\eps_w > 0$, one obtains \eqref{eq:bound-improved} on subdomains
$\tilde\Omega$ with $R\subset\tilde\Omega\subset\subset \Omega$ by
testing the equations with the solutions (multiplied with cut-off
functions) and by considering real and imaginary parts of the results.

(b) Analysis of the scattering problem.  In a scattering problem, one
considers $\Omega = \R^3$ and uses an outgoing wave condition for the
diffracted field as a boundary condition. In this setting, it is not
easy to derive a local $L^2$-bound as in \eqref{eq:apr-bound}. In
\cite{BouchitteSchweizer-Max} and \cite{Lamacz-Schweizer-Max} we
succeeded to work with a contradiction argument: Assuming that the
local $L^2$-norms $t_\eta$ are unbounded along a solution sequence,
the rescaled solution sequence $(E^\eta, H^\eta)/t_\eta$ is locally
bounded; the homogenization result for bounded sequences can be
applied and provides a limit system. In the case $t_\eta\to \infty$,
the limit system has a vanishing incident field and therefore only the
trivial solution. Together with a compactness argument, this provides
a contradiction to the fact that the solutions $(E^\eta,
H^\eta)/t_\eta$ are normalized.

(c) We mention that the trick described in (b) did not work in the
case of obstacles as in \cite{BouchitteSchweizer-Plasmons}, where the
inclusions have a macroscopic length.  The long obstacles in that
contribution result in the geometrical problem that two points in $R$
with distance of order $\eta$ can have a distance of order $1$ in the
metric that is induced by the coefficient $\eps_\eta$.  In the
contribution at hand, the complement of the inclusions is connected in
each periodicity cell, which suggests that the scattering problem
could be solved with Theorem \ref {thm:effective-Max} arguing as in
\cite{BouchitteSchweizer-Max}.

\medskip {\bf Interpretation of the result and comparison to
  experiments.} The result of this contribution is the following: One
can combine in each periodicity cell a resonator and a metallic wire
structure.  The resonator with relative permittivity $\eps_b$ can lead
to a negative magnetic coefficient $\mu^\eff(\omega)$. A negative
relative permittivity $\eps_w$ in the wires can lead to a negative
electric coefficient $\eps^\eff$. Since the wires are thin, the
effects are decoupled: The coefficient $\mu^\eff(\omega)$ is
determined by cell-problems that do not see the wires, the coefficient
$\eps^\eff$ is a result of a negative permittivity in the wires and is
independent of the frequency.

In experiments, the metallic wires can indeed have a large negative
real part, hence our assumptions on the wires are well justified. On
the other hand, the experiments typically use metallic resonators
(e.g.\,split ring resonators); in this set-up, the resonance is not
based on the Mie-resonance of dielectrics, but on capacitor-inductor
interactions.  The methods in this article can also cover this case:
Combining the thin wires with the {\em metallic} split-ring resonators
of \cite{BouchitteSchweizer-Max} or \cite{Lamacz-Schweizer-Max} would
lead to the same result (we emphasize that the coefficient $\kappa$ in
\cite{BouchitteSchweizer-Max} or \cite{Lamacz-Schweizer-Max} can also
be a complex number). We present here the case of dielectrics, since
the analysis of the resonance is much simpler in this case.

Let us high-light why both coefficients can have eigenvalues with
negative real part.  The tensor $\mu^{\eff}$ of \eqref {eq:HhatH-mu} can
be expressed with eigenfunctions $(\fhi_n)_{n\in \N}$ of an eigenvalue
problem in the cell as
\begin{equation}
  \label{eq:mu-eff}
  \mu^{\eff}_{ij} = \delta_{ij} + \sum_{n\in \N} \frac{\eps_b k^2}{\lambda_n - \eps_b k^2}
  \left(e_j\cdot \int_Y \fhi_n\right)
  \left(e_i\cdot \int_Y \fhi_n\right)\,,
\end{equation}
where $k = \omega \sqrt{\eps_0 \mu_0}$ and $(\lambda_n)_{n\in \N}$ are
the eigenvalues. The formula is taken from \cite{BouchitteBourel2009},
it is identical to the present case, since the cell problem is
identical. We sketch the essential arguments that lead from \eqref
{eq:HhatH-mu} to \eqref {eq:mu-eff} in Appendix \ref{sec.appA}. The
formula shows that, if $\eps_b k^2$ is close to an eigenvalue of the
cell problem, then the coefficient $\mu^{\eff}$ can be large in
absolute value and the real part can have both signs.

For the effective permittivity $\eps^{\eff}$, we have the formula
\begin{equation}
  \label{eq:eps-eff-def}
  \eps^{\eff} := A^{\eff} + \pi\alpha^2 \eps_w\,,
\end{equation}
where the (positive) tensor $A^{\rm eff}$ is defined in \eqref
{eq:A-tensor}.  While $A^{\rm eff}$ (and, obviously, also
$\pi\alpha^2$) are always positive, a negative sign of $\Re\, \eps_w$
can lead to a negative real part of (all the eigenvalues of)
$\eps^{\eff}$.

\medskip {\bf Notation.} Spaces of periodic functions are denoted with
the sharp-symbol $\sharp$, e.g.\,as $H^1_\sharp(Y)$.  We use the
wedge-symbol for the cross-product, $v\wedge w := (v_2 w_3 - v_3 w_2,
v_3 w_1 - v_1 w_3, v_1 w_2 - v_2 w_1)$, and the curl-operator
$\curl\,v = \nabla\wedge v$.  Constants $C$ are always independent of
$\eta$, they may change from line to line.

\section{Two-scale limits and cell problems}

In this section, we define the (rescaled) dielectric current $J^\eta$
and consider the two-scale limits $E_0(x,y)$, $H_0(x,y)$, and
$J_0(x,y)$ of the three sequences $H^\eta$, $E^\eta$, and $J^\eta$. We
derive cell problems that determine the two-scale limits up to the
macroscopic averages $E(x)$ and $H(x)$. The procedure uses two-scale
convergence and follows closely the lines of other contributions in
the field, e.g.\,\cite{BouchitteBourel2009, BouchitteSchweizer-Max,
  Lamacz-Schweizer-Max}. The crucial difference regards the inclusion
of thin wires.  In our setting, the cell problems are identical to
those of \cite{BouchitteBourel2009} --- despite the wire
structures. On the one hand, this fact allows to use the results of
\cite{BouchitteBourel2009} on the cell problems. On the other hand,
the effect of the wires decouples from the effect of the dielectric
inclusions. The wires enter only in the derivation of the macroscopic
equation.

\paragraph{The flux $J^\eta$.}

Besides $E^\eta$ and $H^\eta$, we consider a third quantity, namely
the rescaled dielectric field $J^\eta : \Omega\to \C^3$, defined by
\begin{equation}
  \label{eq:J-def}
  J^\eta := \eta\eps_\eta E^\eta\,.
\end{equation}
The definition of $\eps_\eta$ and estimate \eqref {eq:bound-improved}
imply
\begin{equation}
  \label{eq:bound-current}
  \int_{\Omega} |J^\eta|^2 
  = \int_{\Omega\setminus (\Sigma_\eta\cup \Gamma_\eta)} |\eta E^\eta|^2
  + \int_{\Sigma_\eta} |\eps_b|^2 \frac1{\eta^2} |E^\eta|^2
  + \int_{\Gamma_\eta} |\eps_w|^2 \frac1{\eta^2} |E^\eta|^2\le C\,.
\end{equation}
The $L^2(\Omega)$-boundedness of the unknowns $H^\eta$, $E^\eta$, and
$J^\eta$ implies that we can find a sequence $\eta = \eta_i \to 0$ and
two-scale limit functions $E_0(x,y)$, $H_0(x,y)$, and $J_0(x,y)$ such
that $H^\eta$, $E^\eta$, and $J^\eta$ converge weakly in two scales to
the corresponding limit functions (compactness with respect to weak
two-scale convergence).

\subsection{Cell-problem for $E_0$}

\begin{lemma}[Cell-problem for $E_0$]\label{lem:cell-E0}
  Let $(E^\eta, H^\eta)$ be a sequence of solutions of the Maxwell's
  equations as described in Theorem \ref {thm:effective-Max}, and let
  $E_0(x,y)$ be a two-scale limit of $E^\eta$. Then, for almost every
  $x\in R$, the function $E_0 = E_0(x,.)\in H^1_\sharp(Y)$
  satisfies
  \begin{align}
    \label{eq:Ecell1}
    \curl_y\, E_0 &= 0 \text{ in }\, Y\,,\\
    \label{eq:Ecell2}  
    \div_y\, E_0 &= 0 \text{ in }\, Y\setminus \bar\Sigma_Y\,,\\
    \label{eq:Ecell3}
    E_0&=0 \text{ in }\, \Sigma_Y\,.
   \end{align}
   For a given cell-average $E(x) = \meanint_Y E_0(x,y)\, dy$, there
   exists a unique solution $E_0$ of the above problem.
\end{lemma}

{\em Easy parts of the proof and consequences.}  We claim that there
exist three elementary solutions $E^1, E^2, E^3\in H^1_\sharp(Y)$ with
the canonical normalization,
\begin{equation}
  \label{eq:Ej}
  E^j : Y\to \R^3\,,\quad E^j 
  \text{ solves \eqref{eq:Ecell1}--\eqref{eq:Ecell3} and }
  \int_Y E^j = e_j\,.
\end{equation}
The elementary solution $E^j$ can be constructed with a solution of a
Laplace problem: Let $\Theta^j$ be harmonic in $Y\setminus \Sigma_Y$
with $\Theta^j(y) = -y\cdot e_j$ in $\Sigma_Y$. Setting $E^j(y) :=
\nabla_y \Theta^j + e_j$ provides a solution with average $\meanint_Y
E^j = e_j$. Since every solution is given by a gradient of a
potential, the construction implies also the uniqueness of solutions
to prescribed averages.

As a consequence of the lemma, we obtain that the two-scale limit
$E_0(x,y)$ can be written with the help of the three basis
functions as
\begin{equation}
  \label{eq:E0-E}
  E_0(x,y) = \sum_{j=1}^3 E_j(x) E^j(y)\,,
\end{equation}
where $E_j$ coincides with the $j$-th component of the weak limit of
the sequence $E^\eta$.

\begin{proof}[Proof of Lemma \ref {lem:cell-E0}]
  It remains to derive the cell problem and to show that it is not
  affected by the presence of the wire structure. The properties of
  two-scale convergence imply $\eta\,\curl\, E^\eta\to \curl_y\,
  E_0(x,y)$ while, on the other hand, Maxwell's equation \eqref
  {eq:Max1} and boundedness of $H^\eta$ yield $\eta\,\curl\, E^\eta\to
  0$. This provides \eqref{eq:Ecell1}.

  When restricted to $\Sigma_\eta$, the function $\eta^{-1} E^\eta$ is
  $L^2$-bounded by estimate \eqref {eq:bound-improved}. This implies
  \eqref{eq:Ecell3}.

  Outside $\Sigma_\eta \cup \Gamma_\eta$, the electric field $E^\eta$
  has a vanishing divergence. This implies 
  \begin{equation}
    \label{eq:div-0-outside}
    \div_y\, E_0 = 0 \text{ in }\, 
    Y\setminus \left(\bar\Sigma_Y \cup \Gamma_Y \right)
  \end{equation}
  for the one-dimensional wire-centers $\Gamma_Y = \bigcup_{j=1}^3
  \Gamma^j_Y$ in the sense of distributions.  We claim that \eqref
  {eq:div-0-outside} implies \eqref{eq:Ecell2}. In order to prove this
  claim we have to argue with the fact that the set $\Gamma_Y$ has a
  vanishing $H^1(Y)$-capacity.

  We work here only with $p=2$, i.e.\,with the $2$-capacity, for
  definitions see e.g.\,\cite{EvansGariepy}. The one-dimensional
  fibres in dimension $n=3$ have a vanishing capacity. This can be
  concluded from Theorem 3 in \cite{EvansGariepy}, page 254, since
  $n-p = 3-2 = 1$ and the one-dimensional Hausdorff-measure of the
  wires is finite.  The more elementary argument is that points in two
  space dimensions have a vanishing capacity and the corresponding
  cut-off functions can be extended to three-dimensions for straight


  By the definition of a vanishing capacity, for every $\delta>0$,
  there exists a function $\Upsilon_\delta$ with the properties
  \begin{align}
    &\Upsilon_\delta : Y\to \R\text{ is of class } H^1_\sharp(Y)\,,
    \label{eq:Ups-1}\\
    &\Upsilon_\delta \text{ vanishes in a neighborhood of } \Gamma_Y\,,
    \label{eq:Ups-2}\\
    &\| 1 - \Upsilon_\delta \|_{H^1(Y)} \le \delta\,.
    \label{eq:Ups-3}
  \end{align}

  Let now $\fhi\in C^1(Y)$ be an arbitrary test-function with support
  contained in $Y\setminus \bar\Sigma_Y$. We consider $\fhi
  \Upsilon_\delta$, for which \eqref {eq:div-0-outside} and the
  properties of $\Upsilon_\delta$ imply
  \begin{equation}
    0 = \int_Y E_0(y)\cdot \nabla (\fhi \Upsilon_\delta)  
    = \int_Y E_0(y)\cdot \nabla \fhi\ \Upsilon_\delta + O(\delta)
    \to \int_Y E_0(y)\cdot \nabla \fhi
  \end{equation}
  as $\delta\to 0$. We hence obtain \eqref{eq:Ecell2}.
\end{proof}

\subsection{Cell-problem for $(H_0, J_0)$} 

The cell-problem for the magnetic field is obtained as in the case
without wires, we do not have to use vanishing capacity arguments. In
fact, the cell problems are identical to those in the case without
wires\footnote{Our equation \eqref {eq:c3} coincides with equation
  (14) of \cite {BouchitteBourel2009}, except for the negative sign in
  front of the second term; we correct here a typo in (14) of \cite
  {BouchitteBourel2009}. We remark that the analysis of \cite
  {BouchitteBourel2009} proceeds with the correct sign after equation
  (14).}.
\begin{lemma}[Cell-problem for the pair $(H_0, J_0)$]
  \label{lem:cell-H0} Let $(E^\eta, H^\eta)$ be a sequence of
  solutions as in Theorem \ref {thm:effective-Max} and let $J^\eta$ be
  defined by \eqref{eq:J-def}. Let $H_0(x,y)$ and $J_0(x,y)$ be
  two-scale limits of $H^\eta$ and $J^\eta$. Then $H_0(x,.)$
  satisfies, for a.e. $x\in R$,
  \begin{align}
    \curl_y\ H_0 + i\omega\eps_0 J_0 &= 0\text{ in } Y,\label{eq:c1}\\
    \div_y\ H_0 &= 0 \text{ in } Y,\label{eq:c2}
  \end{align}
  while $J_0(x,.)$ satisfies, for a.e. $x\in R$,
  \begin{align}
    \curl_y\ J_0 - i \eps_b \omega\mu_0 H_0 &= 0
    \text{ in } \Sigma_Y,\label{eq:c3}\\
    \div_y\, J_0 &= 0\text{ in } Y,\label{eq:c4}\\
    J_0  & =0 \text{ in } Y\setminus\bar\Sigma_Y.\label{eq:c5}
  \end{align}
\end{lemma}

\begin{proof}
  The Maxwell equations \eqref {eq:Max1}--\eqref {eq:Max2} immediately
  imply relations \eqref{eq:c1}--\eqref{eq:c4}, one only has to use
  the properties of two-scale convergence and the definition of
  $J^\eta$. We note that these five equations hold either on all of
  $Y$ or on $\Sigma_Y$, hence the justification is as in the problem
  without wires.

  Relation \eqref{eq:c5} is a consequence of the a priori estimate: On
  the set $\Omega_\eta := \Omega\setminus (\Sigma_\eta\cup
  \Gamma_\eta)$, the function $\eta^{-1} J^\eta = \eps_\eta E^\eta =
  E^\eta$ is bounded. More precisely: $\eta^{-1} J^\eta
  \one_{\Omega_\eta}$ is bounded in $L^2(\Omega)$ and hence $J^\eta
  \one_{\Omega_\eta}$ converges strongly to $0$ in $L^2(\Omega)$. This
  implies for the two-scale limit function the relation
  \begin{equation}
    \label{eq:J-vanish}
    J_0(x,y) = 0 \text{ for almost every }  
    y\in Y\setminus (\Sigma_Y \cup \Gamma_Y)\,,
  \end{equation}
  for almost every $x\in R$. This relation is identical to
  \eqref{eq:c5}, since $\Gamma_Y$ is a set of vanishing measure.
\end{proof}

It was shown in \cite{BouchitteBourel2009} that, with an appropriate
normalization, the cell-problem \eqref{eq:c1}--\eqref{eq:c5} has a
unique solution. A natural normalization is to prescribe the geometric
average $\oint H$ of the cell solution, we define this average in the
next subsection.  The normalized solutions are denoted by $H^j =
H^j(y)$, $j = 1,2,3$,
\begin{equation}
  \label{eq:Hj}
  H^j, J^j : Y\to \C^3\,,\quad (H^j, J^j) 
  \text{ solves \eqref{eq:c1}--\eqref{eq:c5} and }
  \oint H^j = e_j\,.
\end{equation}
As a result, the two-scale limit $H_0(x,.)$ of $H^\eta$ can be written
as a linear combination of the three shape functions $H^j$. Denoting
the coefficients by $\hat H_j(x)$, we write
\begin{equation}
  \label{eq:H0-expansion}
  H_0(x,y) = \sum_{j=1}^3 \hat H_j(x) H^j(y)\,.  
\end{equation}
It remains to establish the relation between the weak limit $H(x)$ and
the coefficients $\hat H_j(x)$. As the weak limit of $H^\eta$, the
function $H(x)$ coincides with the $Y$-average of $H_0(x,.)$.  Taking
the $Y$-average of \eqref {eq:H0-expansion}, we obtain
\begin{equation}\label{eq:HhatH-H}
  H(x) = \int_Y H_0(x,y)\, dy 
  = \sum_{j=1}^3 \hat H_j(x) \int_Y H^j(y)\, dy
  = \mu^{\eff}\, \hat H(x)
\end{equation}
by setting
\begin{equation}
  \label{eq:HhatH-mu}
  \mu_{i,j}^{\eff} := e_i \cdot \int_Y H^j(y)\,dy\,.
\end{equation}
Relation \eqref {eq:HhatH-H} was used in Theorem \ref
{thm:effective-Max} to define the field $\hat H$.

\subsection{Geometric averaging}
\label{ssec.geometric-av}

There are two possibilities to average a function $u:Y\to \C^3$. The
standard averaging procedure is the volumetric average, given by the
integral over $Y$.  In a geometric averaging procedure, we seek for
the typical value of an integral over a curve that connects opposite
sides. The concept has been made precise with the definition of a {\em
  circulation vector} in \cite{BouchitteBourel2009}: To a vector field
$u\in X = \{ u\in H^1_\sharp(Y; \C^3)| \curl_y u = 0 \text{ on }
Y\setminus \Sigma_Y\}$, we associate the circulation vector $\oint u
\in \C^3$. It is defined by the property
\begin{equation}
  \label{eq:circulation}
  \int_Y u\cdot \fhi = 
  \left(  \oint  u \right) \cdot \left(  \int_Y  \fhi\right) \,
\end{equation}
for every test-function $\fhi\in L^2(Y, \C^3)$ with $\fhi(y) = 0$ for
almost every $y\in \Sigma_Y$ and with $\nabla\cdot \fhi = 0$ in the
sense of distributions on $Y$, more precisely: $\int_{Y} \fhi\cdot
\nabla U = 0$ for every $U\in H^1_\sharp(Y)$.

We must verify that \eqref {eq:circulation} indeed defines a vector
$\oint u \in \C^3$. To this end we must show that, for every function
$\fhi$ with average $\int_Y \fhi = e_i$, the expression $\int_Y u\cdot
\fhi$ has the same value. By linearity, it is sufficient to show that
$\int_Y u\cdot \fhi_0$ vanishes for every test-function $\fhi_0$ with
a vanishing average.  

As a preparation, we note that the conditions on the test-functions
$\fhi$ imply that, for every function $U\in H^1_\sharp(Y)$, there
holds $\int_{Y\setminus \Sigma_Y} \fhi\cdot \nabla U = \int_{Y}
\fhi\cdot \nabla U = 0$.  Since $Y\setminus \Sigma_Y$ is simply
connected, a curl-free function $u$ can be written, up to an additive
constant $u_0\in \C^3$, as a gradient of a periodic function $U$, i.e.
$u(y) = u_0 + \nabla U(y)$ for $y\in Y\setminus \Sigma_Y$.  We
therefore find $\int_Y u\cdot \fhi_0 = \int_{Y\setminus \Sigma} (u_0 +
\nabla U)\cdot \fhi_0 = u_0\cdot \int_{Y} \fhi_0 = 0$.

To understand the definition of $\oint u \in \C^3$ through \eqref
{eq:circulation} better, let $u$ be a function in $X$ and let $\gamma$
be a curve that does not touch $\bar\Sigma_Y$ and that connects the
lower and the upper face of $Y$, i.e. the side $\{ y_3 = 0\}$ with $\{
y_3 = 1\}$, in points that are identified by periodicity. If $u$ is
continuous and $\gamma$ is differentiable, there holds, for the $3$-rd
component,
\begin{equation}
  \label{eq:circulation-lineint}
  e_3\cdot \oint u = \int_\gamma u\cdot \tau\,,
\end{equation}
where $\tau$ is the tangential vector along $\gamma$. The formula
expresses that every component of the circulation vector of $u$
coincides with the corresponding line integral of $u$. Due to $\curl_y
u = 0$ in $Y\setminus \Sigma_Y$, the line integral is independent of
the curve $\gamma$ by the Stokes' theorem (we recall that $Y\setminus
\Sigma_Y$ is simply connected).

Let us demonstrate \eqref {eq:circulation-lineint} with the help of
special test-functions. The test functions are constructed with the
help of cylinders that are similar to the wires in our homogenization
setting. Focussing on the third component, we use $\Gamma_\delta :=
B_\delta((y_1,y_2)) \times [0,1)$.  Choosing the point $(y_1, y_2)\in
(0,1)^2$ sufficiently close to the boundary and $\delta>0$ small, we
acchieve that $\Gamma_\delta$ does not touch $\bar\Sigma_Y$. In order
to evaluate the circulation vector, we consider the function $\fhi(y)
= e_3\,\one_{\Gamma_\delta}(y)$ that points in the third coordinate
direction and vanishes outside the cylinder. The function $\fhi$ has a
vanishing divergence and it vanishes on $\Sigma_Y$, it is therefore a
valid test-function.  Its volume average is $\int_Y \fhi =
\vol_3(\Gamma_\delta)\, e_3 =: V_3\, e_3$.  We can therefore calculate
the left hand side of \eqref {eq:circulation-lineint} as
\begin{align*}
  e_3\cdot \oint u 
  = \frac1{V_3}
  \int_Y u\cdot \fhi 
  = \frac1{V_3}
  \int_{B_\delta((y_1,y_2))} \left(\int_{\gamma_{(z_1,z_2)}} u\cdot e_3\, dz_3\right)\, dz_1\,dz_2
  = \int_\gamma u\cdot \tau\,.
\end{align*}
In the first equality, we used the definition \eqref {eq:circulation},
in the second equality, we used Fubini's theorem and decomposed the
integral over the cylinder into an integral over the cross-section and
integrals over the corresponding lines. In the third equality, we used
that all line integrals coincide and that $V_3 = \vol_3(\Gamma_\delta)
= \int_{B_\delta((y_1,y_2))} 1$.  We have thus obtained \eqref
{eq:circulation-lineint}.

\paragraph{Application to the two-scale limits $E_0$ and $H_0$.}
We can now understand better the connection \eqref {eq:HhatH-H}
between $\hat H(x)$ and $H(x)$. Confronting the normalization of
$H^j(y)$ in \eqref {eq:Hj} with the definition of $\mu^\eff$ in \eqref
{eq:HhatH-mu},
\begin{equation}\label{eq:H-normalization-rewrite}
  \oint H^j = e_j\,,
  \qquad\quad
  \mu^{\eff}\cdot e_j = \int_Y H^j\,,
\end{equation}
we see that the matrix $\mu^\eff$ corresponds to the factor between
$Y$-averages and geometric averages of cell solutions.  Instead, for
the cell-solutions $E^j$, the two averaging procedures coincide,
\begin{equation}\label{eq:E-averages-coincide}
  \oint E^j = \int_Y E^j = e_j\,,
\end{equation}
which is a consequence of $\curl_y E^j = 0$ on $Y$: all line integrals
coincide, even if the curves intersect $\Sigma_Y$. This fact yields
$\hat E(x) = E(x)$.

\section{Macroscopic equations}
\label{sec.macro-equations}

In this section, we obtain the macroscopic equation \eqref {eq:eff2}.
The starting point for the derivation is the Maxwell equation \eqref
{eq:Max2} and our knowledge about the two-scale limits $E_0$ and
$H_0$.  The main idea is to construct special test-functions for
\eqref {eq:Max2}. A first type of test-functions has a vanishing curl
in the whole domain, a second type of test-functions has a specified
curl in one of the wires.

\subsection{The $\eta$-dependent test-functions avoiding wire
  integrals}

We start the construction by defining, for $j\in \{1,2,3\}$ fixed, a
potential $\Theta^j_\eta : Y\to \R$ of class $\Theta^j_\eta \in
H^1_\sharp(Y)$.  We set $Y_\eta^* := Y\setminus (\bar\Sigma_Y\cup
\overline{\bigcup_{i=1}^3 \Gamma^{\eta,i}_Y})$ and define
$\Theta^j_\eta$ as the solution of
\begin{align}
  \Theta^j_\eta(y) &= -e_j\cdot y\text{ in } 
  \Sigma_Y\cup \bigcup_{i\neq j} \Gamma^{\eta,i}_Y\,,\label{eq:Theta-1}\\
  \Theta^j_\eta(y) &= 0\text{ in } \Gamma^{\eta,j}_Y\,,\label{eq:Theta-2}\\
  \Delta \Theta^j_\eta(y) &= 0\text{ in } Y_\eta^*\,.\label{eq:Theta-3}
\end{align}
The construction allows to define $\vth^j_\eta : Y\to \R^n$ with the
gradient of $\Theta^j_\eta$,
\begin{equation}
  \label{eq:vth-def}
  \vth^j_\eta = \nabla \Theta^j_\eta + e_j\,.
\end{equation}
This provides a function $\vth^j_\eta$ that satisfies
\begin{align}
  \curl\, \vth^j_\eta &= 0\text{ in } Y\,,\\
  \div\, \vth^j_\eta &= 0\text{ in } Y_\eta^*\,,\\
  \vth^j_\eta &= 0\text{ in } \Sigma_Y\cup \bigcup_{i\neq j} \Gamma^{\eta,i}_Y\,,\label{eq:vth-2}\\
  \vth^j_\eta &= e_j\text{ in } \Gamma^{\eta,j}_Y\,,
\end{align}
with the normalization $\int_Y \vth^j_\eta = e_j$.  We note that
$\vth^j_\eta$ is closely related to the solution $E^j:Y\to \R^n$ of
the electrical cell problem (which does not see the wires). In fact,
we have the following convergence result.

\begin{lemma}\label{lem:vth-convergence}
  There holds
  \begin{equation}
    \label{eq:vth-convergence}
    \vth^j_\eta \to E^j \text{ in } L^2(Y)\text{ as } \eta\to 0.
  \end{equation}
\end{lemma}

\begin{proof}
  We fix the index $j\in \{1,2,3\}$ and use variational arguments. The
  potentials $\Theta^j_\eta$ minimize the Dirichlet functional
  \begin{equation*}
    A(\Theta) := \int_Y |\nabla \Theta|^2\quad \text{ on }\quad
    X_\eta := \left\{ \Theta_\eta\in H^1_\sharp(Y,\R) | \Theta_\eta \text{ satisfies 
      \eqref {eq:Theta-1}
      and  \eqref {eq:Theta-2}} \right\}\,.
  \end{equation*}
  The potential $\Theta^j$ (that provides the cell solution $E^j =
  \nabla_y \Theta^j + e_j$) minimizes the Dirichlet functional $A$ on
  the set $X_* := \left\{ \Theta\in H^1_\sharp(Y,\R) | \Theta = -e_j\cdot
    y\text{ in } \Sigma_Y \right\}$. From $X_\eta\subset X_*$ we
  immediately obtain $A(\Theta^j) \le A(\Theta^j_\eta)$.

  \smallskip Our aim is to show the convergence of the energies as
  $\eta\to 0$,
  \begin{equation}
    \label{eq:aim-normconv}
    A(\Theta^j_\eta) \to A(\Theta^j)\,.    
  \end{equation}
  We observe that \eqref {eq:aim-normconv} implies the lemma: The
  functions $\Theta^j_\eta$ are bounded in $H^1_\sharp(Y)$, hence we
  can select a weakly convergent subsequence. The weak lower
  semi-continuity of $A$ on $X_*$ together with \eqref
  {eq:aim-normconv} implies that every weak limit of $\Theta^j_\eta$
  minimizes $A$, it therefore coincides with the (unique) minimizer
  $\Theta^j$. Since \eqref {eq:aim-normconv} implies additionally the
  convergence of norms, we can conclude the strong convergence of the
  sequence. By uniqueness of the limit, the whole sequence converges.

  \medskip The convergence \eqref {eq:aim-normconv} follows if we
  show, for arbitrary $\delta_0>0$, that the relation
  $A(\Theta^j_\eta) \le A(\Theta^j) + \delta_0$ holds for every
  sufficiently small $\eta$.  In order to obtain this inequality, we
  construct, for arbitrary $\delta_0>0$, a comparison function $\tilde
  \Theta^j_\delta\in X_\eta$ with $A(\Theta^j_\eta) \le A(\tilde
  \Theta^j_\delta) \le A(\Theta^j) + \delta_0$.

  We use the cut-off functions $\Upsilon_\delta$ of \eqref
  {eq:Ups-1}--\eqref {eq:Ups-3}, more precisely: The function
  $\Upsilon_\delta$ is periodic, vanishes on the subset of all wires
  $\Gamma_Y^\eta = \bigcup_{i=1}^3 \Gamma_Y^{\eta,i}$, equals $1$ on
  $\Sigma_Y$ and is close to $1$ in the $H^1_\sharp(Y)$-norm. We
  additionally use a function $\Upsilon^*_\delta$ with similar
  properties: $\Upsilon^*_\delta$ is periodic, vanishes on the subset
  $\bigcup_{i\neq j}^3 \Gamma_Y^{\eta,i}$ of two wires, is equal to
  $1$ on $\Gamma_Y^{\eta,j} \cup \Sigma_Y$ and close to $1$ in the
  $H^1_\sharp(Y)$-norm.  We construct
  \begin{equation}
    \label{eq:Theta-Ups}
    \tilde \Theta^j_\delta := \Theta^j\, \Upsilon_\delta 
    + (-e_j\cdot y) (1- \Upsilon^*_\delta)\,.
  \end{equation}
  The functions $\tilde \Theta^j_\delta$ are periodic on $Y$; we
  emphasize that this fact is only true since we do {\em not} set
  $\tilde \Theta^j_\delta = -y_j$ on the $j$-th wire.  Furthermore,
  the comparison functions satisfy $\tilde \Theta^j_\delta \in X_\eta$
  for $\eta>0$ sufficiently small. They have the Dirichlet energy
  \begin{align*}
    A(\tilde \Theta^j_\delta) := \int_Y |\nabla \tilde \Theta^j_\delta|^2
    = \int_Y |\nabla (\Theta^j\, \Upsilon_\delta) - \nabla (y_j (1-
    \Upsilon^*_\delta))|^2\,.
  \end{align*}
  The smallness assumption \eqref {eq:Ups-3} and the boundedness of
  $\Theta^j$ (due to the maximum principle) imply
  \begin{align*}
    A(\tilde \Theta^j_\delta) \le \int_Y |\nabla (\Theta^j\, \Upsilon_\delta) |^2 
    + O(\delta)
     \le \int_Y |\nabla \Theta^j|^2\, |\Upsilon_\delta |^2 + O(\delta)
     \le A(\Theta^j) + O(\delta)\,.
  \end{align*}
  Because of the minimality of $\Theta^j_\eta$ we have
  $A(\Theta^j_\eta) \le A(\tilde \Theta^j_\delta)$ and obtain
  therefore $A(\Theta^j_\eta) \le A(\Theta^j) + O(\delta)$. This
  yields \eqref {eq:aim-normconv} and thus the statement of the lemma.
\end{proof}

\begin{proposition}\label{prop:macro-convergences-I}
  Let $(E^\eta, H^\eta)$ be a sequence of solutions as in Theorem \ref
  {thm:effective-Max}. Let $\fhi\in C^1_c(\Omega; \R)$ be arbitrary
  and let $\vth^j_\eta$ be the special test-functions defined
  above. Then there holds, as $\eta\to 0$,
  \begin{align}
    \int_\Omega H^\eta(x)\cdot (\vth^j_\eta(x/\eta)\wedge \nabla\fhi(x))\, dx &\to 
    \int_\Omega\int_Y H_0(x,y)\cdot 
    (E^j(y)\wedge\nabla\fhi(x))\,dy\,dx\,,\label{eq:Heta-lim}\\
    \int_{\Omega\setminus \Gamma^j_\eta} \eps_\eta E^\eta(x) 
    \vth^j_\eta(x/\eta) \fhi(x)\, dx &\to 
    \int_\Omega\int_Y E_0(x,y)\cdot E^j(y)\, \fhi(x)\,dy\, dx\,.
    \label{eq:Eeta-lim}
  \end{align}
\end{proposition}

\begin{proof}
  Relation \eqref {eq:Heta-lim} follows from the strong two-scale
  convergence of $\vth^j_\eta(x/\eta)$ to $E^j(y)$ that was shown in
  Lemma \ref {lem:vth-convergence}. For the convergence of products of
  strongly and weakly two-scale convergent sequences we refer to
  \cite{Allaire1992}.

  Concerning \eqref{eq:Eeta-lim} we note that $\eps_\eta = 1$ holds in
  $\Omega_\eta = \Omega\setminus (\Sigma_\eta\cup \Gamma_\eta)$ and
  that $\vth^j_\eta$ vanishes on $(\Sigma_\eta\cup
  \Gamma_\eta)\setminus\Gamma^j_\eta$. Again, the strong two-scale
  convergence of $\vth^j_\eta(x/\eta)$ to $E^j(y)$ implies the
  convergence to the double integral.
\end{proof}

\subsection{The calculation of wire integrals}

With the last proposition, we are almost in the position to derive the
macroscopic limit equations: The left hand side of \eqref
{eq:Heta-lim} coincides with $-\curl\, H^\eta$, tested against a
suitable test-function. The left hand side of \eqref {eq:Eeta-lim}
coincides with $\eps_\eta\, E^\eta$, tested against the same
test-function. Since the two quantities are related by the Maxwell
equation \eqref{eq:Max2}, a comparison of the two right hand sides can
provide the missing effective equation that relates $\curl\, H$ with
$E$.

There is one missing piece: the left hand side of \eqref {eq:Eeta-lim}
does not include the integration of $\eps_\eta E^\eta\cdot e_j$ over
the $j$-th wire, $\Gamma_\eta^j$. Its limit cannot be calculated as in
the previous subsection, since $\Gamma_\eta^j$ connects opposite sides
of the cell in a periodic way. Technically: $\Theta^j_\eta(y) =
-e_j\cdot y$ cannot be demanded in \eqref {eq:Theta-1}, since this
function is not periodic in $y_j$.

We therefore calculate the missing integral with a different
approach. The result is, in some sense, not very surprising: The wire
integrals of $E^\eta$ converge to line integrals of the two-scale
limit function $E_0$ (up to the factor $\pi\alpha^2$ that measures the
volume of the cross section of the wire). This stability of line
integrals over $E^\eta$ in a two-scale limit process is a consequence
of the fact that the curl of $E^\eta$ is controlled.

\begin{proposition}\label{prop:macro-convergences-II}
  Let $(E^\eta, H^\eta)$ be a sequence of solutions as in Theorem \ref
  {thm:effective-Max} and let $\fhi\in C^1_c(\Omega; \R)$ be
  arbitrary. Then there holds, for $j= 1,2,3$,
  \begin{align}
    \int_{\Gamma_\eta^j} \frac{1}{\eta^2} E^\eta(x)\cdot e_j\, \fhi(x)\, dx 
    &\to \pi\alpha^2 \int_\Omega E_j(x)\, \fhi(x)\, dx\,.
     \label{eq:Eeta-lim-Gamma}
  \end{align}
\end{proposition}

Even though the result is suggestive, the proof of Proposition \ref
{prop:macro-convergences-II} requires some preparation.  As in the
proof of Proposition \ref {prop:macro-convergences-I}, we must
construct special test-functions.

\paragraph{Geometry and notation.}
We will construct an $\eta$-dependent sequence of test-functions
$g^\eta:Y\to \R^3$ in several steps. To simplify notation we assume
the following: By choice of other coordinates, the unit cell is $Y =
(-\tfrac12, \tfrac12)^3$, the coordinates are gathered as $y = (y_1,
y_2, y_3) = (\tilde y, y_3)$. We consider $j=3$, i.e. that wire inside
$Y$ that runs in $e_3$-direction. We exploit that the wire is straight
and assume for convenience that it has the central line $\tilde y =
0$, i.e. $\Gamma_Y^3 = \{ y\in Y| \tilde y = 0\}$ and
$\Gamma_Y^{\eta,3} = B_{\alpha\eta}(\Gamma_Y^3) = B_{\alpha\eta}(0)
\times [-\tfrac12, \tfrac12)$. We choose a radius $\delta>0$
sufficiently small such that $B_\delta(\Gamma_Y^3)$ does not intersect
the other sub-structures, $\Gamma_Y^{\eta,1}$, $\Gamma_Y^{\eta,2}$,
and $\Sigma_Y$. In the following, we only consider $\eta$ with
$\alpha\eta<\delta/2$.

\paragraph{Construction of test-functions.}
We start with a scalar function
\begin{equation}
  \label{eq:psi-def}
  \psi_\eta : [\alpha\eta,\infty)\to \R,\qquad 
  \psi_\eta(r) := \alpha\eta \log(r) \rho_\delta(r)\,,
\end{equation}
where $\rho_\delta: [0,\infty)\to \R$ is a smooth and monotone cut-off
function with $\rho_\delta(r) = 1$ for every $r<\delta/2$ and
$\rho_\delta(r) = 0$ for every $r\ge\delta$.  The function $\psi_\eta$
allows to define the vector potential $\Theta_\eta$:
\begin{equation}
  \label{eq:Theta-def}
  \Theta_\eta : Y\to \R,\qquad 
  \Theta_\eta(y) := \psi_\eta(|\tilde y|)\, e_3\,.
\end{equation}
With the help of the potential $\Theta_\eta$, we can finally define
the test-function $g_\eta$ by setting
\begin{equation}
  \label{eq:g-def}
  g_\eta : Y\to \R^3,\qquad 
  g_\eta(y) := 
  \begin{cases}
    -\curl_y \Theta_\eta\quad &\text{ for } |\tilde y| > \alpha\eta\,,\\[3mm]
    {\displaystyle  \frac{1}{\alpha\eta}} \vecttre{-y_2}{y_1}{0}
    \quad &\text{ for } |\tilde y| \le \alpha\eta\,.
  \end{cases}
\end{equation}

For another explanation of the ideas, let us describe the construction
as follows: We consider a (truncated) two-dimensional fundamental
solution outside the disc, and define $g_\eta$ outside the wire as its
rotated gradient. Inside the wire, we define $g_\eta$ as a rigid
rotation.

\paragraph{Properties of the test-functions.}
We start with the observation that $g_\eta$ is continuous, coinciding
with a normalized tangential vector on the cylinder surface. Indeed,
using the two-dimensional tangential vector $\tau = (-y_2,
y_1)/|\tilde y| \equiv (-y_2, y_1,0)/|\tilde y|$, we have defined
$g_\eta(y) = (-y_2, y_1,0)/(\alpha\eta) = \tau$ for $r = |\tilde y| =
\alpha\eta$. On the other hand, for $|\tilde y| > \alpha\eta$, we find
with $r = |\tilde y|$
\begin{align*}
  g_\eta(y) = -\curl_y \Theta_\eta = \del_r \psi_\eta(r) \tau\,,\
  \text{ and }\ \del_r \psi_\eta(r) = \alpha\eta \frac1{r} \rho_\delta(r)
  + \alpha\eta \log(r) \del_r \rho_\delta(r)\,.
\end{align*}
In particular, because of $\del_r \rho_\delta = 0$ in a neighborhood
of $r = \alpha\eta$, there holds $g_\eta(y) \to \tau$ as $r = |\tilde
y|\to \alpha\eta$.

\smallskip The function $g_\eta$ is bounded in the unit cube, $\|
g_\eta \|_\infty \le 2$ independently of $0<\eta<\eta_0$ for some
small $\eta_0>0$. This property follows by inspection of the
expression for $\del_r \psi_\eta(r)$, in which the first term is
bounded by $1$, the second term is bounded by $C \eta
|\log(\eta)|$. Furthermore, there holds
\begin{equation}
  \label{eq:L2-limit-g-eta}
  \| g_\eta \|_{L^2(Y)} \to 0\,.  
\end{equation}
Indeed, for the integral over the small region $|\tilde
y|<\alpha\eta$, we can exploit boundedness of $g_\eta$. For the
integral over $|\tilde y| > \alpha\eta$, we calculate for the first
term $\int_\eta^\delta (\eta/r)^2\, r\,dr = \eta^2 \log(\delta/\eta)
\to 0$ as $\eta\to 0$. The second term is uniformly bounded by $C \eta
|\log(\eta)|$ and we conclude \eqref {eq:L2-limit-g-eta}.

\smallskip In the next step, we calculate the curl of $g_\eta$. For
$|\tilde y| < \alpha\eta$, there holds
\begin{equation}\label{eq:curl-g-inn}
  \curl_y\, g_\eta(y) = \frac1{\alpha\eta} \nabla^\perp\cdot \vectdue{-y_2}{y_1} e_3
  = \frac2{\alpha\eta}\, e_3\,.
\end{equation}
Instead, for $\alpha\eta<|\tilde y|<\delta/2$, we find, using $\div\,
\Theta_\eta = 0$,
\begin{align*}
  \curl_y\, g_\eta(y) 
  = \Delta_{(y_1,y_2)} \left[\alpha\eta\log(|\tilde y|)\right]\, e_3
  = 0\,.
\end{align*}
Finally, for $|\tilde y|\ge \delta/2$, there holds
\begin{equation}\label{eq:curl-g-out}
  \frac1\eta \curl_y\, g_\eta(y) 
  = \Delta_{(y_1,y_2)} \left[\alpha \log(|\tilde y|) 
    \rho_\delta(|\tilde y|)\right]\, e_3
  =: G_\delta(y)\,.
\end{equation}
We extend $G_\delta$ by setting $G_\delta(y) = 0$ for $|\tilde y| <
\delta/2$. We then find $\eta^{-1} \curl_y\, g_\eta(y) = G_\delta(y)$
for $|\tilde y| > \alpha\eta$. It is the main point of our construction, that
this function is non-singular in the limit $\eta\to 0$, despite the
factor $\eta^{-1}$.

As a last step in this preparation, we observe that $G_\delta(y)$
points always in direction $e_3$ and that it is independent of
$y_3$. Its average can be calculated with Stokes' theorem, using the
normal vector $\nu = (y_1, y_2, 0)/(\alpha\eta)$ on the cylinder
surface and the fact that $\nu\wedge\tau = e_3$:
\begin{align}\label{eq:G-average}
  \int_Y G_\delta(y)\cdot e_3\, dy 
  &= \int_{|\tilde y| > \alpha\eta} \frac1\eta \curl_y\, g_\eta(y)\cdot e_3\, dy 
  = - \int_{|\tilde y| = \alpha\eta} \frac1\eta\, \nu\wedge g_\eta(y)\cdot  e_3\, dS(y)\\
  &= -\int_{|\tilde y| = \alpha\eta} \frac1\eta\, dS(y)
  = -2\pi\alpha\,.\nonumber
\end{align}

\paragraph{Convergence of wire averages.}

With the help of the above oscillatory test-functions we can now prove
the convergence result.

\begin{proof}[Proof of Proposition \ref{prop:macro-convergences-II}]
  From \eqref {eq:curl-g-inn}, we know the curl of $g_\eta$ inside the
  wire. Transforming into the $x$-variables, we have $\curl_x \left[
    g_\eta(x/\eta)\right] = 2/(\alpha\eta^2) e_3$ in $\Gamma^3_\eta$.
  We can therefore express the wire integral with the help of the curl
  of an approriately designed test-function and proceed with a
  straight-forward calculation. In the equation marked with
  ``\eqref{eq:curl-g-out}'' below, we perform an integration by parts
  of the curl-operator, in the limit marked with
  ``\eqref{eq:L2-limit-g-eta}'' we exploit the boundedness of $\curl_x
  E^\eta$, which follows from the first Maxwell equation and the
  $L^2$-boundedness on $H^\eta$.
  \begin{align*}
    &\frac2{\alpha}\int_{\Gamma^3_\eta} \frac{1}{\eta^2} E^\eta(x)\cdot
    e_3\,\fhi(x)\, dx \stackrel{\eqref{eq:curl-g-inn}}{=}
    \int_{\Gamma^3_\eta} E^\eta(x)\cdot \curl_x \left[
      g_\eta(x/\eta)\right]\,\fhi(x)\,dx\\
    &\quad = \int_{\Omega} E^\eta(x)\cdot \curl_x \left[
      g_\eta(x/\eta)\right]\,\fhi(x)\,dx - \int_{\Omega\setminus
      \Gamma^3_\eta} E^\eta(x)\cdot \curl_x \left[
      g_\eta(x/\eta)\right]\,\fhi(x)\,dx\\
    &\quad\stackrel{\eqref{eq:curl-g-out}}{=} \int_{\Omega} \curl_x
    E^\eta(x)\cdot \left[ g_\eta(x/\eta)\right]\,\fhi(x)\,dx 
    - \int_{\Omega} E^\eta(x)\cdot \left[
      g_\eta(x/\eta)\right]\wedge \nabla_x \fhi(x)\,dx\\
    &\qquad\qquad - \int_{\Omega\setminus \Gamma^3_\eta} E^\eta(x)\cdot
    G_\delta(x/\eta)\,\fhi(x)\,dx\displaybreak[2]\\
    &\quad\stackrel{\eqref{eq:L2-limit-g-eta}}{\to} -
    \int_{\Omega}\int_Y E_0(x,y)\cdot G_\delta(y)\,\fhi(x)\,dy\,dx\displaybreak[2]\\
    &\quad\stackrel{\eqref{eq:circulation}}{=} 
    - \int_{\Omega}\left(\oint E_0(x,y)\cdot e_3\right) 
    \int_Y G_\delta(y)\,dy\, \fhi(x)\,dx\displaybreak[2]\\
    &\quad\stackrel{\eqref{eq:E-averages-coincide}}{=} 
    - \int_{\Omega} E_3(x) \int_Y G_\delta(y)\cdot e_3\,dy\, \fhi(x)\,dx\displaybreak[2]\\
    &\quad\stackrel{\eqref{eq:G-average}}{=} 
    2\pi\alpha \int_{\Omega} E_3(x)\, \fhi(x)\,dx\,.
  \end{align*}
  This calculation shows the limit \eqref {eq:Eeta-lim-Gamma} and
  hence Proposition \ref {prop:macro-convergences-II}.
\end{proof}

\subsection{Derivation of the macroscopic equations}

\paragraph{Limit process in \eqref {eq:Max1}.} We can take the
distributional limit of \eqref {eq:Max1} and obtain, in the limit
$\eta\to 0$,
\begin{align}
  \curl\, E &\leftarrow
  \curl\, E^\eta = i\omega\mu_0 H^\eta
  \to i\omega\mu_0 H = i\omega\mu_0\hat \mu \hat H\,.
\end{align}
We recall that the last equation is a consequence of the definition of
$\hat \mu$ and $\hat H$. The above distributional limit equation
already provides \eqref {eq:eff1}, the first of the two effective
equations.

\paragraph{Re-writing of the two-scale limit integrals.}

It remains to conclude the second effective equation, \eqref
{eq:eff2}.  We will obtain this equation from \eqref {eq:Max2},
exploiting Propositions \ref {prop:macro-convergences-I} and \ref
{prop:macro-convergences-II}. In order to prepare the calculation, we
re-write terms that have been obtained in Proposition \ref
{prop:macro-convergences-I}.

We define the coefficient matrix
\begin{equation}\label{eq:A-tensor}
  A_{i,j}^{\eff} := \int_Y E^i(y)\cdot E^j(y)\, dy\,.
\end{equation}
With this definition, we can write the $Y$-integral on the right hand
side of \eqref {eq:Eeta-lim}, for $x\in R$, as
\begin{align}\label{eq:Aeff-use}
  \int_Y E_0(x,y)\cdot E^j(y)\, dy
  = \sum_{i=1}^3 \left(\int_Y E^i(y)\cdot E^j(y)\, dy\right) E_i(x)
  = (A^{\eff} E(x))_j\,.
\end{align}

To calculate the right hand side of \eqref {eq:Heta-lim}, we use the
expansion \eqref {eq:H0-expansion} of $H_0$ and the definition of the
circulation vector: The function $E^j\wedge e_k$ is a test-function
which vanishes on $\Sigma_Y$ and which has a vanishing divergence;
regarding the latter we recall $E^j = \nabla \Theta^j + e_j$ which
implies, for $k=3$, $\nabla\cdot (E^j(y)\wedge e_3) = \del_1 E^j_2 -
\del_2 E^j_1 = \del_1 \del_2 \Theta^j - \del_2 \del_1 \Theta^j = 0$.
This allows to express the integral of a product with the circulation,
\begin{equation}\label{eq:H0-circ-hatH}
  \begin{split}
    &\int_Y H_0(x,y) \cdot (E^j(y)\wedge e_k)\, dy
    \stackrel{\eqref{eq:H0-expansion}}{=}   
    \sum_{i=1}^3 \hat H_i(x)  \int_Y H^i(y) \cdot (E^j(y)\wedge e_k)\, dy\\
    &\qquad \stackrel{\eqref{eq:circulation}}{=} 
    \sum_{i=1}^3 \hat  H_i(x)\, \left( \oint H^i \right) 
    \cdot \left( \int_Y E^j(y) \wedge e_k \right) 
    = \sum_{i=1}^3 \hat H_i(x)\, e_i\cdot (e_j\wedge e_k)\,.
  \end{split}
\end{equation}
With this preparation, we can now perform the limit process.

\paragraph{Limit process in \eqref {eq:Max2}.}

In order to perform the limit $\eta\to 0$ in \eqref {eq:Max2}, we use
an oscillating test-function. We choose a smooth function $\fhi
:\Omega\to \R$ with compact support and fix $j\in \{1,2,3\}$. We
consider $\psi_\eta(x) = \vth^j_\eta(x/\eta)\,\fhi(x)$ with
$\vth^j_\eta$ from \eqref {eq:vth-def}. Then the second Maxwell
equation \eqref {eq:Max2} yields
\begin{equation}\label{eq:effective-2}
  \int_\Omega \curl\, H^\eta \cdot \psi_\eta = -i\omega\eps_0 \int_\Omega 
  \eps_\eta E^\eta\cdot \psi_\eta\,.
\end{equation}
It remains to evaluate the limits of both sides of \eqref
{eq:effective-2}. We start with the left hand side. In the subsequent
calculation we use first integration by parts, then $\curl_y\,
\vth^j_\eta(y) = 0$. In the limit process we exploit \eqref
{eq:Heta-lim} of Proposition \ref {prop:macro-convergences-I}:
\begin{align*}
  &\int_\Omega \curl\, H^\eta \cdot \psi_\eta 
  = \int_\Omega H^\eta\cdot \curl\, \psi_\eta
  = - \int_\Omega H^\eta(x)\cdot (\vth^j_\eta(x/\eta) \wedge \nabla\fhi(x))\, dx\\
  &\qquad \to -\int_\Omega\int_Y H_0(x,y)\cdot   (E^j(y)\wedge\nabla\fhi(x))\,dy\,dx
  \stackrel{\eqref{eq:H0-circ-hatH}}{=} 
  -\int_\Omega \hat H(x)\cdot (e_j\wedge \nabla \fhi(x))\, dx   \\
  &\qquad  = \int_\Omega \hat H(x)\cdot \curl\, (\fhi(x)\, e_j)\, dx
  = \int_\Omega (\curl\, \hat H)\cdot e_j\ \fhi \,.
\end{align*}
We now calculate the right hand side of \eqref {eq:effective-2}. In
the first equality, we use that $\vth^j_\eta(x/\eta)$ vanishes on
$\Sigma_\eta$ and on all $\Gamma^i_\eta$ with $i\neq j$, and that it
coincides with $e_j$ in $\Gamma^j_\eta$.  The limit process $\eta\to
0$ for the two integrals has been prepared in \eqref {eq:Eeta-lim} and
\eqref {eq:Eeta-lim-Gamma}.
\begin{align*}
  &\int_\Omega \eps_\eta E^\eta \cdot\psi_\eta 
  = \int_{\Omega\setminus (\Sigma_\eta\cup \Gamma_\eta)} 
  \eps_\eta E^\eta \cdot \vth^j_\eta(x/\eta) \fhi(x)\, dx
  + \int_{\Gamma_\eta^j} \eps_\eta E^\eta(x)\cdot e_j \fhi(x)\, dx \\
  &\qquad \to \int_\Omega\int_Y E_0(x,y)\cdot E^j(y)\, \fhi(x)\,dy\, dx
  + \pi\alpha^2 \eps_w \int_\Omega E_j(x)\, \fhi(x)\, dx\\
  &\qquad \stackrel{\eqref{eq:Aeff-use}}{=}
  \int_\Omega A^{\eff} E(x)\cdot e_j\ \fhi(x)\, dx
  + \pi\alpha^2 \eps_w \int_\Omega E(x)\cdot e_j\ \fhi(x)\, dx\,.
\end{align*}

Since $j\in \{1,2,3\}$ and $\fhi = \fhi(x)$ are arbitrary, we obtain
from \eqref {eq:effective-2}
\begin{align*}
  \curl\, \hat H = -i\omega\eps_0 (A^{\eff} + \pi\alpha^2 \eps_w)\, E(x)\,.
\end{align*}
This is exactly the effective equation \eqref {eq:eff2}, since we have
defined the effective permittivity in \eqref {eq:eps-eff-def} as
$\eps^{\eff} = A^{\eff} + \pi\alpha^2 \eps_w$.

\section*{Conclusions}

We have investigated Maxwell's equations in a periodic material with
small periodicity length $\eta>0$. The permeability is set to $1$, the
permittivity is assumed to have extreme values of order $O(\eta^{-2})$
in the periodic inclusions, it is $1$ outside the inclusions. Two
types of inclusions are present: bulk inclusions and wire inclusions.
The dielectric bulk inclusions have an impact on the effective
permeability $\mu^\eff$, an effect that has been studied before in
\cite{BouchitteBourel2009}.  In our setting, the cell-problems for
$\mu^\eff$ are identical to those of \cite{BouchitteBourel2009} and
the study of the spectral problem is already available. Negative
coefficients $\mu^\eff$ are possible due to resonance effects.  We
mention that our approach could also be carried out with metallic
inclusions ($\eps_b$ with a negative real part), if one constructs
resonators with a split ring structure as in \cite
{BouchitteSchweizer-Max} or \cite{Lamacz-Schweizer-Max}.

The new feature in the present work is the network of thin wires. We
have seen that this network contributes to the effective permittivity
$\eps^\eff$. The formula \eqref {eq:eps-eff-def} for $\eps^\eff$ is
frequency independent, the relevant new contribution is explicitely
given as $\pi\alpha^2 \eps_w$ (and is not given through a cell
problem). The wires do {\em not} create a negative permittivity
through some resonance effect, but merely through an averaging
procedure: $\pi\alpha^2\eta^2$ is the volume of the wires, $\eps_w
\eta^{-2}$ is the permittivity in the wires.

Nevertheless, let us emphasize that we observe here an effect that is
more involved than some simple averaging: Only the connectedness of
the wires across cells makes the effect possible (i.e.: the topology
of the wires). Indeed, if $\Gamma^j_\eta$ did not connect opposite
sides, the test function $\vth_\eta^j$ could be constructed such that
\eqref{eq:vth-2} holds also in $\Gamma^j_\eta$. In that case, the wire
had no effect in the averaged law.

\appendix
\section{Bulk-resonance and the formula for $\mu^\eff$}
\label{sec.appA}

In order to derive formula \eqref {eq:mu-eff} for $\mu^\eff$, one has
to calculate the $Y$-averages of the solutions $H^j$ to the
cell-problem of Lemma \ref {lem:cell-H0}. We briefly sketch the
arguments leading to \eqref {eq:mu-eff}, following
\cite{BouchitteBourel2009}. The underlying concept of describing the
cell-problem for $H^j$ with a bilinear form on a suitable Hilbert
space has been used aleady in \cite{BouchitteSchweizer-Max} (which was
written earlier than \cite{BouchitteBourel2009}), but the useful
concept of geometric averaging was only introduced with
\cite{BouchitteBourel2009}.

One considers the Hilbert space $X_0 := \{ u\in H^1_\sharp(Y)\, |\,
\curl\, u = 0\text{ in } Y\setminus \Sigma_Y, \oint u = 0\}$ and the
bilinear form $b_0(u,v) := \int_Y \{ \curl\, u\cdot \curl\, \bar v +
\div\, u\cdot \div\, \bar v \}$. The solutions $H^j$ of \eqref {eq:Hj}
are of the form $H^j = e_j + u_j$ where $u_j\in X_0$ is determined by
the variational equation ($k^2 = \omega^2 \eps_0 \mu_0$)
\begin{equation*}
  b_0(u_j, v) - k^2 \eps_b \int_Y u_j \cdot \bar v 
  = k^2 \eps_b \int_Y e_j\cdot \bar v\qquad
  \forall v\in X_0\,.
\end{equation*}
The symmetric bilinear form $b_0$ is coercive, it hence defines an
operator $B_0$ that has a compact self-adjoint resolvent on
$L^2(Y)$. The orthonormal eigenfunctions $(\fhi_n)_{n\in \N}$ to
eigenvalues $(\lambda_n)_{n\in \N}$ of $B_0$ allow to express
solutions as $u_j = \sum_n c_{j,n} \fhi_n$ with $c_{j,n} = \eps_b k^2
(\lambda_n - \eps_b k^2)^{-1} \int_Y e_j\cdot \fhi_n$. Definition
\eqref {eq:HhatH-mu} of the effective tensor $\mu^\eff$ provides
\eqref {eq:mu-eff}.

\vspace*{-2mm}
\subsection*{Acknowledgements}

Support of both authors by DFG grant Schw 639/6-1 is greatfully
acknowledged.

\bibliographystyle{abbrv} 
\bibliography{lit_negative}

\def\ocirc#1{\ifmmode\setbox0=\hbox{$#1$}\dimen0=\ht0 \advance\dimen0
  by1pt\rlap{\hbox to\wd0{\hss\raise\dimen0
  \hbox{\hskip.2em$\scriptscriptstyle\circ$}\hss}}#1\else {\accent"17
  #1}\fi}\def\cprime{$'$}
\begin{thebibliography}{10}

\bibitem{Allaire1992}
G.~Allaire.
\newblock Homogenization and two-scale convergence.
\newblock {\em SIAM J. Math. Anal.}, 23(6):1482--1518, 1992.

\bibitem{Arbogast-1052874}
T.~Arbogast, J.~Douglas, Jr., and U.~Hornung.
\newblock Derivation of the double porosity model of single phase flow via
  homogenization theory.
\newblock {\em SIAM J. Math. Anal.}, 21(4):823--836, 1990.

\bibitem{BouchitteBourel2009}
G.~Bouchitt{\'e}, C.~Bourel, and D.~Felbacq.
\newblock Homogenization of the 3{D} {M}axwell system near resonances and
  artificial magnetism.
\newblock {\em C. R. Math. Acad. Sci. Paris}, 347(9-10):571--576, 2009.

\bibitem{CRAS}
G.~Bouchitt{\'e} and D.~Felbacq.
\newblock Homogenization near resonances and artificial magnetism from
  dielectrics.
\newblock {\em C. R. Math. Acad. Sci. Paris}, 339(5):377--382, 2004.

\bibitem{MR2262964}
G.~Bouchitt{\'e} and D.~Felbacq.
\newblock Homogenization of a wire photonic crystal: the case of small volume
  fraction.
\newblock {\em SIAM J. Appl. Math.}, 66(6):2061--2084, 2006.

\bibitem{BouchitteSchweizer-Cloak}
G.~Bouchitt{\'e} and B.~Schweizer.
\newblock Cloaking of small objects by anomalous localized resonance.
\newblock {\em Quart. J. Mech. Appl. Math.}, 63(4):437--463, 2010.

\bibitem{BouchitteSchweizer-Max}
G.~Bouchitt{\'e} and B.~Schweizer.
\newblock Homogenization of {M}axwell's equations in a split ring geometry.
\newblock {\em Multiscale Model. Simul.}, 8(3):717--750, 2010.

\bibitem{BouchitteSchweizer-Plasmons}
G.~Bouchitt{\'e} and B.~Schweizer.
\newblock Plasmonic waves allow perfect transmission through sub-wavelength
  metallic gratings.
\newblock {\em Netw. Heterog. Media}, 8(4):857--878, 2013.

\bibitem{DLS-MR3191584}
T.~Dohnal, A.~Lamacz, and B.~Schweizer.
\newblock Bloch-wave homogenization on large time scales and dispersive
  effective wave equations.
\newblock {\em Multiscale Model. Simul.}, 12(2):488--513, 2014.

\bibitem{DLS-2}
T.~Dohnal, A.~Lamacz, and B.~Schweizer.
\newblock Dispersive homogenized models and coefficient formulas for waves in
  general periodic media.
\newblock {\em Asymptotic Analysis}, 93(1-2):21--42, 2015.

\bibitem{EvansGariepy}
L.~C. Evans and R.~F. Gariepy.
\newblock {\em Measure theory and fine properties of functions}.
\newblock Studies in Advanced Mathematics. CRC Press, Boca Raton, FL, 1992.

\bibitem{MR1444123}
D.~Felbacq and G.~Bouchitt{\'e}.
\newblock Homogenization of a set of parallel fibres.
\newblock {\em Waves Random Media}, 7(2):245--256, 1997.

\bibitem{KohnShipman}
R.~Kohn and S.~Shipman.
\newblock Magnetism and homogenization of micro-resonators.
\newblock {\em Multiscale Modeling \& Simulation}, 7(1):62--92, 2007.

\bibitem{Kohn-LSW}
R.~V. Kohn, J.~Lu, B.~Schweizer, and M.~I. Weinstein.
\newblock A variational perspective on cloaking by anomalous localized
  resonance.
\newblock {\em Comm. Math. Phys.}, 328(1):1--27, 2014.

\bibitem{Lamacz-Schweizer-Max}
A.~Lamacz and B.~Schweizer.
\newblock Effective {M}axwell equations in a geometry with flat rings of
  arbitrary shape.
\newblock {\em SIAM J. Math. Anal.}, 45(3):1460--1494, 2013.

\bibitem{Pendry2004}
D.~Smith, J.~Pendry, and M.~Wiltshire.
\newblock Metamaterials and negative refractive index.
\newblock {\em Science}, 305:788--792, 2004.

\bibitem{Veselago1968}
V.~Veselago.
\newblock The electrodynamics of substances with simultaneously negative values
  of $\varepsilon$ and $\mu$.
\newblock {\em Soviet Physics Uspekhi}, 10:509--514, 1968.

\bibitem{MR2029130}
N.~Wellander and G.~Kristensson.
\newblock Homogenization of the {M}axwell equations at fixed frequency.
\newblock {\em SIAM J. Appl. Math.}, 64(1):170--195 (electronic), 2003.

\end{thebibliography}

\end{document}